\documentclass[12pt]{article}
\usepackage{amsfonts,amsmath,amsthm,url}
\title{Criterion for linear independence of functions}
\author{I. V. Romanovski\footnote{ramanowski@tut.by},\\{\it Belarusian State University}}
\begin{document}
\maketitle
\begin{abstract}
Using a generalization of forward elimination, it is proved that
functions $f_1,\ldots,f_n:X\to\mathbb{A}$, where $\mathbb{A}$ is a
field, are linearly independent if and only if there exists a
nonsingular matrix $\left[f_i(x_j)\right]$ of size $n$, where
$x_1,\ldots,x_n\in X$.
\end{abstract}

\noindent{\bf Keywords:} Gaussian elimination, system of functions,
linearly independent functions.

\bigskip\noindent{\bf MSC2000:} 15A03.

\section{Introduction}

Suppose we are dealing with a separable kernel (see, e.g., \cite{1},
p. 4) $K:[a,b]\times[a,b]\to\mathbb{R}$ of an integral operator, i.
e.
\begin{equation}K(t,s)\equiv\sum\limits_{j=1}^nT_j(t)S_j(s),\label{nonzeroK}\end{equation}
where
\begin{equation}T_1,\ldots,
T_n,S_1,\ldots,S_n:[a,b]\to\mathbb{R}.\label{functionsTS}\end{equation}
Suppose $K\ne0$. Then we may consider each of the systems
$\{T_1,\ldots,T_n\}$, $\{S_1,\ldots,S_n\}$ linearly independent.
Indeed, starting from an expression of kind \eqref{nonzeroK}, we
consequently reduce the number of items while it is needed.

Assume that we need to express the functions \eqref{functionsTS} in
terms of $K$. In order to do it, we find such points (a proof of
existence and a way of finding will follow)
\begin{equation}t_1,\ldots,t_n,s_1,\ldots,s_n\in[a,b],\label{pointsts}\end{equation} that the square matrices
$T=[T_j(t_i)]$, $S=[S_j(s_i)]$ of size $n$ are nonsingular, write
out the identities
$$
K(t_i,s)\equiv\sum\limits_{j=1}^nT_j(t_i)S_j(s),~K(t,s_i)\equiv\sum\limits_{j=1}^nS_j(s_i)T_j(t),~i=1,\ldots,n,
$$
i. e.
$$
\left[\begin{array}{c}K(t_1,s)\\\vdots\\K(t_n,s)
\end{array}\right]\equiv T\cdot\left[\begin{array}{c}S_1(s)\\\vdots\\S_n(s)
\end{array}\right],~
\left[\begin{array}{c}K(t,s_1)\\\vdots\\K(t,s_n)
\end{array}\right]\equiv S\cdot\left[\begin{array}{c}T_1(s)\\\vdots\\T_n(s)
\end{array}\right],
$$
and obtain the desired expressions
\begin{equation}
\left[\begin{array}{c}S_1(s)\\\vdots\\S_n(s)
\end{array}\right]\equiv T^{-1}\cdot\left[\begin{array}{c}K(t_1,s)\\\vdots\\K(t_n,s)
\end{array}\right],~
\left[\begin{array}{c}T_1(s)\\\vdots\\T_n(s)
\end{array}\right]\equiv S^{-1}\cdot\left[\begin{array}{c}K(t,s_1)\\\vdots\\K(t,s_n)
\end{array}\right].
\label{expressions}\end{equation} The formulas \eqref{expressions}
let one, for example, prove smoothness of the functions
\eqref{functionsTS} if $K$ is smooth.

The existence of such points \eqref{pointsts} seems doubtless: if we
considered the set $\{1,\ldots,m\}$ instead of $[a,b]$, the matrices
$[T_j(i)]$ and $[S_j(i)]$ of size $m\times n$ would be of full rank
(see \cite{2}) and would correspondingly have nonsingular
submatrices $T=[T_j(t_i)]$ and $S=[S_j(s_i)]$ of size $n$. But what
about a strict proof?

\section{Results}

Let $X$ be a nonempty set, let
$\mathbb{A}$ be a field and let
$$\mathbb{F}_{m,n}=\left\{F=\left[\begin{array}{c}f_{11}~\ldots~f_{1n}\\.\:\:.\:\:.\:\:.\:\:.\:\:.\:\:.\\f_{m1}~\ldots~f_{mn}\end{array}\right]:X\to\mathbb{A}_{m,n}\right\},$$
where
$\mathbb{A}_{m,n}\!=\!\left\{\!A\!=\!\left[\!\begin{array}{c}\alpha_{11}~\ldots~\alpha_{1n}\\.\;\;.\;\;.\;\;.\;\;.\;\;.\;\;.\\\alpha_{m1}~\ldots~\alpha_{mn}\end{array}\!\right]:\alpha_{ij}\in\mathbb{A},~i=1,\ldots,m,~j=1,\ldots,n\right\}$,
$m=1,2,\ldots$, $n=1,2,\ldots$.

\newtheorem{lma}{Lemma}
\begin{lma}Let the entries of the column
$\,f\,=\,\left[\,f_i\,\right]\,\in\,F_{n,1}\,$ be linearly
independent functions and let the matrix
$A=\left[\alpha_{ij}\right]\in\mathbb{A}_{n,n}$ be nonsingular. Then
$g=\left[g_i\right]=Af$ is a column of linearly independent
functions.\end{lma}

\begin{proof}Let $\beta_1,\ldots,\beta_n\in\mathbb{A}$ and
$\sum\limits_{i=1}^n\beta_ig_i=0$. We are going to prove that
\begin{equation}\beta_1=\ldots=\beta_n=0.\label{zeromu}\end{equation}
Indeed, let's denote the row
$\left[\beta_j\right]\in\mathbb{A}_{1,n}$ by $\beta^T$ and multiply
the equity $g=Af$ by $\beta^T$ from the left. We have
$0=(\beta^TA)f$. Note that $\beta^T A\in F_{1,n}$ is a row. Since
the entries of $f$ are linearly independent functions then $\beta^T
A=0$. The matrix $A$ is nonsingular, therefore the equities
\eqref{zeromu} hold.
\end{proof}

Let $x=\left(x_1,\ldots,x_n\right)\in X^n$, $f\in\mathbb{F}_{n,1}$
and let $f(x)$ denote the matrix
$\left[f_i(x_j)\right]\in\mathbb{A}_{n,n}$. Obviously,
\begin{equation}(Af)(x)=A\cdot f(x)\label{assoc}\end{equation} for any
$A\in\mathbb{A}_{n,n}$.

\begin{lma}Let the matrices
$A,(Af)(x)\in\mathbb{A}_{n,n}$ be nonsingular. Then the matrix
$f(x)$ is nonsingular.
\end{lma}

\begin{proof}It follows from the formula \eqref{assoc} that $\det
\left(f(x)\right)={\det\left((Af)(x)\right)\over\det A}\ne0$.
\end{proof}

Further, given a column $f=\left[f_i\right]\in \mathbb{F}_{n,1}$ of
linearly independent functions, we will find such a vector $x\in
X^n$ and such a matrix
\begin{equation}A=\left[\begin{array}{cccc}1&0&\ldots&0\\\alpha_{21}&1&\ldots&0\\\vdots&&\ddots&\\\alpha_{n1}&\alpha_{n2}&&1\end{array}\right]\in\mathbb{A}_{n,n}\label{L}\end{equation} that
$(Af)(x)$ is of kind
\begin{equation}\left[\begin{array}{cccc}\beta_{11}&\beta_{12}&\ldots&\beta_{1n}\\0&\beta_{22}&\ldots&\beta_{2n}\\\vdots&&\ddots&\vdots\\0&0&&\beta_{nn}\end{array}\right]\in\mathbb{A}_{n,n},~\beta_{ii}\ne0,~i=1,\ldots,n.\label{U}\end{equation}
Because of nonsingularity of matrices $A,(Af)(x)\in\mathbb{A}_{n,n}$
and lemma 2, the matrix $f(x)$ will be nonsingular.

\newtheorem{thma}{Theorem}
\begin{thma}Let the entries of the column
$\,f=\left[\,f_i\,\right]\in F_{n,1}\,$ be linearly independent
functions. Then there exists such a vector $x\in X^n$ and such a
matrix $A\in\mathbb{A}_{n,n}$ of kind \eqref{L} that the matrix
$(Af)(x)$ is of kind \eqref{U}.\end{thma}

\begin{proof}Let's use mathematical induction on $n$.
\begin{enumerate}
\item Let $n=1$. Then there exists such $x_1\in X$ that
$f_1(x_1)\ne0$, because otherwise $f_1=0$ and hence the system
$\{f_1\}$ is linearly dependent.
\item Let $n>1$. As in the
case 1, we find such $x_1\in X$ that $f_1(x_1)\ne0$. Let
$$
M=\left[\begin{array}{cccc}1&0&\ldots&0\\-{f_2(x_1)\over
f_1(x_1)}&1&&0\\\vdots&&\ddots&\\-{f_n(x_1)\over
f_1(x_1)}&0&&1\end{array}\right],
$$
$g=[g_i]=Mf$. Because of nonsingularity of the matrix $M$ and lemma
1, $g$ is a column of linearly independent functions. Also
$$g_1=f_1,~g_i(x_1)=-{f_i(x_1)\over f_1(x_1)}f_1(x_1)+f_i(x_1)=0,~i=2,\ldots,n.$$ Let's
consider the following block partition
$g=\left[\begin{array}{c}f_1\\\hline\\[-14pt]\tilde
g\end{array}\right]$, where $\tilde g\in\mathbb{F}_{n-1,1}$. Since
any subsystem of a linearly independent system is itself linearly
independent, the entries of $\tilde g$ are linearly independent
functions. Moreover, \begin{equation}\tilde
g(x_1)=0.\label{zerog}\end{equation}

Let's find such a vector $\tilde x=\left(\tilde x_1,\ldots,\tilde
x_{n-1}\right)\in X^{n-1}$ and such a matrix $\tilde
B\in\mathbb{A}_{n-1,n-1}$ of kind \eqref{L} that $(\tilde B \tilde
g)(\tilde x)$ is of kind \eqref{U}. Let $x=\left(x_1,\tilde
x_1,\ldots,\tilde x_{n-1}\right)\in X^n$,
$B=\left[\begin{array}{c|c}1&0\\\hline\\[-12pt]0&\tilde
B\end{array}\right]\in\mathbb{A}_{n,n}$. Obviously, $B$ is of kind
$\eqref{L}$. Note, that $(Bg)(x)=\left[\begin{array}{c}f_1\\\hline\\[-12pt]\tilde B\tilde
g\end{array}\right](x)=\left[\begin{array}{c|c}f_1(x_1)&f_1(\tilde
x)\\\hline\\[-12pt](\tilde B\tilde g)(x_1)&(\tilde B\tilde g)(\tilde x)\end{array}\right]$ is of
kind \eqref{U}, because $f(x_1)\ne 0$, $(\tilde B\tilde
g)(x_1)=\tilde B\cdot\tilde g(x_1)=0$, by \eqref{assoc},
\eqref{zerog}, and $(\tilde B \tilde g)(\tilde x)$ is of kind
\eqref{U}.

Let $A=BM$. Then $A\in\mathbb{A}_{nn}$, $A$ is of kind \eqref{L} (as
a product of matrices of such kind) and $(Af)(x)=(BMf)(x)=(Bg)(x)$
is of kind \eqref{U}.
\end{enumerate}
\end{proof}

\begin{thma}[{\bf criterion for linear independence of
functions}]The functions $f_1,\ldots,f_n:X\to\mathbb{F}$ are
linearly independent if and only if there exists such
$\left(x_1,\ldots,x_n\right)\in X^n$ that the matrix
$\left[f_i(x_j)\right]\in\mathbb{A}_{n,n}$ is nonsingular.\end{thma}

\begin{proof}Suppose that the entries of the column
$f=[f_i]\in\mathbb{F}_{n,1}$ are linearly independent functions.
Then, by theorem 1 and lemma 2, there exists such $x\in X^n$ that
$f(x)$ is nonsingular.

Now let $x\in X^n$ and let the matrix
$\left[f_i(x_j)\right]\in\mathbb{A}_{n,n}$ be nonsingular. Assume
that $\alpha_1,\ldots,\alpha_n\in\mathbb{A}$ and
$\sum\limits_{i=1}^n\alpha_if_i=0$. In particular, we have
\begin{equation}\sum\limits_{i=1}^n\alpha_if_i(x_j)=0,~j=1,\ldots,n.\label{SLAE}\end{equation}
Considering \eqref{SLAE} a nondegenerate system of linear algebraic
equations in unknowns $\alpha_1,\ldots,\alpha_n$ we conclude that
$\alpha_1=\ldots=\alpha_n=0$. Thus the functions $f_1,\ldots,f_n$
are linearly independent.
\end{proof}

\newtheorem{ex}{Example}
\begin{ex}
Let $\mathbb{A}=\mathbb{C}$ and let $f_1,f_2,f_2\in\mathbb{F}$.
Suppose that $|f_1(x_1)|>|f_1(x_2)|+|f_1(x_3)|$,
$|f_2(x_2)|>|f_2(x_1)|+|f_2(x_3)|$ and $|f_3(x_3)|>
|f_3(x_1)|+|f_3(x_2)|$. Then the matrix
$\left[f_i(x_j)\right]\in\mathbb{C}_{3,3}$ is diagonally dominant
(see \cite{3}) and therefore nonsingular. Thus, by theorem 2, the
functions $f_1,f_2,f_3$ are linearly independent.
\end{ex}

\begin{ex}
Let $\mathbb{A}=\mathbb{C}$ and let $f_1,f_2,f_2\in\mathbb{F}$.
Suppose that $|f_1(x_1)|>|f_2(x_1)|+|f_3(x_1)|$,
$|f_2(x_2)|>|f_1(x_2)|+|f_3(x_2)|$ and $|f_3(x_3)|>
|f_1(x_3)|+|f_2(x_3)|$. Analogously to the previous example, the
matrix $\left[f_i(x_j)\right]\in\mathbb{C}_{3,3}$ is nonsingular and
thus the functions $f_1,f_2,f_3$ are linearly independent.
\end{ex}

\begin{ex}
Let $X=Y\times Z$ and let $f_1,\ldots,f_n\in\mathbb{F}$. Suppose
that $z^\ast\in Z$ and $\varphi_i:Y\to\mathbb{A}$ ($i=1,\ldots,n$)
are such linearly independent functions that $\varphi_i(y)\equiv
f_i(y,z^\ast)$, $i=1,\ldots,n$. Then, by theorem 2, there exists
such $(y_1,\ldots,y_n)\in Y^n$ that the matrix
$\left[\varphi_i(y_j)\right]\in\mathbb{A}_{n,n}$ is nonsingular.
Note that this matrix equals
$\left[f_i(x_j)\right]\in\mathbb{A}_{n,n}$, where
$x_j=(y_j,z^\ast)$, $j=1,\ldots,n$. Thus, by theorem 2, the
functions $f_1,\ldots,f_n$ are linearly independent.
\end{ex}

Taking into account the notion of rank of a system of vectors (see
\cite{4}, p. 52) and, in particular, of rank of a system of
functions in the linear space $\mathbb{F}$, we prove a more general
theorem.

\begin{thma}Let $f_1,\ldots,f_n\in\mathbb{F}$. Then
$$\mathrm{rank}\{f_1,\ldots,f_n\}=\max\mathop{\mathrm{rank}}\limits_{x_1,\ldots,x_n\in
X}\left[f_i(x_j)\right].$$\end{thma}

\begin{proof}Let $r=\mathrm{rank}\{f_1,\ldots,f_n\}$ and let
$r'=\max\mathop{\mathrm{rank}}\limits_{x_1,\ldots,x_n\in
X}\left[f_i(x_j)\right]$. Note that $r,r'\ge0$.

Let's prove that $r'\ge r$. Indeed, if $r=0$, the inequality $r'\ge
r$ holds. Let $r>0$. Then there exists a subset
$\{f_{k_1},\ldots,f_{k_{r}}\}\subseteq\{f_1,\ldots,f_n\}$ of $r$
linearly independent functions. Also, by theorem 2, there exists
such $(x_1,\ldots,x_r)\in X^r$ that the matrix
$\left[f_{k_i}(x_j)\right]\in\mathbb{A}_{r,r}$ is nonsingular. Hence
$r'\ge r$.

Now let's prove that $r\ge r'$. If $r'=0$, then $r\ge r'$. Let
$r'>0$. Then there exists such a subset
$\{f_{k_1},\ldots,f_{k_{r}}\}\subseteq\{f_1,\ldots,f_n\}$ and such
$(x_1,\ldots,x_{r'})\in X^{r'}$ that the matrix
$\left[f_{k_i}(x_j)\right]\in\mathbb{A}_{r',r'}$ is nonsingular.
Therefore, by theorem 2, the functions $f_{k_1},\ldots,f_{k_{r'}}$
are linearly independent. Thus $r\ge r'$.
\end{proof}

\end{document}